\documentclass[12pt,reqno]{amsart}
\usepackage{fullpage}

\usepackage{tikz}
\usetikzlibrary{matrix,arrows,decorations.pathmorphing}
\usepackage{tikz-cd}
\usetikzlibrary{arrows}

\usepackage{amsmath}
\usepackage{mathrsfs}

\usepackage{epsfig}
\usepackage{amsmath,amsfonts,amsthm,amscd, amssymb}
\usepackage{latexsym}
\usepackage{enumerate}

\newcommand*{\rom}[1]{\expandafter\@slowromancap\romannumeral #1@}
\newcommand{\Proof}[1]{\begin{proof} #1 \end{proof}}
\newtheorem{thm}{Theorem}[section]
\newtheorem*{theorem*}{Theorem}
\newtheorem*{question*}{Question}

\newcommand{\Thm}[1]{\begin{thm} #1 \end{thm}}
\newtheorem{pro}[thm]{Proposition}
\newcommand{\Pro}[1]{\begin{pro} #1\end{pro}}
\newtheorem{cor}[thm]{Corollary}
\newcommand{\Cor}[1]{\begin{cor} #1 \end{cor}}
\newtheorem{lem}[thm]{Lemma}
\newcommand{\Lem}[1]{\begin{lem} #1 \end{lem}}

\newtheorem{defn}[thm]{Definition}

\newtheorem{rem}[thm]{Remark}
\newcommand{\Rem}[1]{\begin{rem} #1 \end{rem}}
\newtheorem{ex}[thm]{Example}

\newtheorem{obs}[thm]{Observation}

\newtheorem{note}[thm]{Note}

\newtheorem{nota}[thm]{Notation}

\newcommand{\catS}{\mathbf{S}}

\newcommand{\catU}{\mathbf{U}}
\newcommand{\catK}{\mathbf{K}}

\newcommand{\catE}{\mathbf{E}}

\newcommand{\tightoverset}[2]{%
\mathop{#2}\limits^{\vbox to -.5ex{\kern-0.25ex\hbox{$#1$}\vss}}}

\newcommand{\Hom}{\text{Hom}}

\newcommand{\Un}{\text{Un}}
\newcommand{\Fix}{\text{Fix}}

\newcommand{\join}{\ast}

\newcommand{\idy}{\text{id}}

\newcommand{\rep}{\text{rep}}
\newcommand{\Qd}{\text{Qd}}

\newcommand{\ev}{\text{ev}}

\newcommand{\Sec}{\text{Sec}}
\newcommand{\hSec}{\text{hSec}}
\newcommand{\Th}{\text{Th}}

\newcommand{\Map}{\text{Map}}

\newcommand{\Dim}{\text{Dim}}

\newcommand{\SL}{\text{SL}}

 \newcommand{\normal}{\ensuremath{\triangleleft}} 

\newcommand{\set}[1]{\ensuremath{ \lbrace #1 \rbrace }}

 \newcommand{\sm}{\wedge}

\newcommand{\FF}{\mathbb{F}}
\newcommand{\ZZ}{\mathbb{Z}}

\newcommand{\sS}{\mathbb{S}}

\newcommand{\aA}{\mathcal{A}}

\newcommand{\Ss}{\mathcal{S}}

\newcommand{\bZ}{\mathbb Z}
\newcommand{\bF}{\mathbb F}

\usepackage{chngcntr}
\makeatletter
\let\c@equation=\c@subsubsection

\makeatother

\title{Dimension functions for spherical fibrations}
\date{}
 \address{Department of Mathematics, University of Western Ontario, London ON  N6A 5B7}
\email{cokay@uwo.ca}
\address{Department of Mathematics, Bilkent University, Ankara, 06800, Turkey}
\email{yalcine@fen.bilkent.edu.tr}
\author{C\.{i}han Okay and Erg\"{u}n Yal\c{c}in}

\begin{document}

\begin{abstract} 
Given a spherical fibration $\xi$ over the classifying space $BG$  of a finite group we define a dimension function   for the $m$--fold fiber join of $\xi$ where $m$ is some large positive integer. We show that the dimension functions satisfy the Borel-Smith conditions when $m$ is large enough. As an application we prove that there exists no spherical fibration over the classifying space of $\Qd(p)= (\ZZ/p)^2\rtimes\SL_2(\ZZ/p)$ with $p$--effective Euler class, generalizing the result of \cite{Unl04} about group actions on finite complexes homotopy equivalent to a sphere. We have been informed that this result will also appear in \cite{AGapp} as a corollary of a previously announced program on homotopy group actions due to Jesper Grodal.
\end{abstract}  

\maketitle

\section{Introduction}  
This paper is motivated by a conjecture about group actions on products of spheres due to Benson and Carlson \cite{BC87}. The conjecture states that the maximal rank of an elementary abelian $p-$group contained in a finite group is at most $k$ if and only if there exists a finite free $G$-CW-complex $X$ homotopy equivalent to a product of spheres $\sS^{n_1}\times \sS^{n_2}\times \cdots \times \sS^{n_k}$.   When $k=1$ this conjecture is proved by Swan \cite{Swa60}. The next case $k=2$ is proved by Adem and Smith \cite{AS01} for  finite groups that do not involve  $\Qd(p)= (\ZZ/p)^2\rtimes\SL_2(\ZZ/p) $ for any odd prime $p$. 

An important technique developed in \cite{AS01} for constructing free actions starts with a spherical fibration over $BG$ whose Euler class is $p-$effective  and uses fiber joins to construct a free action on a finite complex homotopy equivalent to a product of two spheres.  One source of such a spherical fibration is a finite $G$-CW-complex $X\simeq \sS^n$ with rank one isotropy. \"{U}nl\" u \cite{Unl04} proved that for $G=\Qd(p)$ there exists no such finite $G$-CW-complex. The main goal of this paper is to extend this result by showing that there exists no spherical fibration over $BG$ with $p-$effective Euler class when $G$ is $\Qd(p)$. We also  show that $\Qd(p)$ cannot act freely on a finite  complex homotopy equivalent to $\sS^n\times \sS^n$. However, the case of the Benson-Carlson conjecture where the dimensions of the spheres are different remains open. 

Given a spherical fibration $\xi: E\to BG$ over $BG$ with fibers $\sS^n$, there is an infinite-dimensional free $G$-space $X_\xi$, defined as the pull-back of $\xi$ along the universal fibration $EG\rightarrow BG$, such that the Borel construction $EG\times_GX_\xi\to BG$ is fiber homotopy equivalent to $\xi$.
 Two $G$-spaces $X$ and $Y$ are said to be $hG$-equivalent if there is a zig-zag sequence of $G$-maps between $X$ and $Y$ that are weak equivalences (non-equivariantly). The fibre homotopy classes of $n$-dimensional spherical fibrations over $BG$ are in one-to-one correspondence with $hG$-equivalence classes of $G$-spaces that are homotopy equivalent to $\sS^n$  (see \S \ref{hG} for details). We will use this correspondence throughout the paper without giving further explanations. 

Let $G$ be a $p-$group and $X$ be a finite dimensional $G$-CW-complex. We write $H(-)$ for mod-$p$ cohomology.
Classical Smith theory says that if $H(X)\cong H(\sS^n)$ for some $n$ then the fixed points  $X^G$ also has the mod-$p$ cohomology of a sphere. A systematic way of studying fixed point subspaces is to define dimension functions $n_X$   by setting
$$
H(X^{K})\cong H(\sS^{n_X(K)-1})
$$
for a subgroup $K\leq G$. It is a fundamental fact that $n_X$ satisfies certain properties called the Borel-Smith conditions. Smith theory fails for infinite-dimensional complexes. The problem is that up to homotopy every action can be made free by taking a product with the universal contractible free $G$--space $EG$. One way around this problem is to consider homotopy fixed points $X^{hK}=\Map(EK,X)^K$ instead of ordinary fixed points. An important algebraic tool for studying cohomology of homotopy fixed points is Lannes' $T-$functor and its variant the $\Fix$ functor. Here a technical point is that $X$ needs to be replaced by its Bousfield-Kan $p$--completion $X^\wedge_p$ and the theory only works for elementary abelian $p$--groups.   Then a theorem of Lannes' relates the mod-$p$ cohomology of homotopy fixed points $(X^\wedge _p )^{hV}$ to the algebraically defined object $\Fix _V (H(X_{hV}))$ for an elementary abelian $p$--subgroup $V$ in $G$, where $X_{hV}=EV\times_VX$. 

Lannes' theory can be applied  under certain conditions. We show that these conditions can be satisfied by replacing a given $G$--space $X  \simeq \sS^n$ with the $p$--completion of its $m$-fold join  
$$
X[m]= (\underbrace{X \ast\cdots \ast X }_m)^\wedge_p.
$$
For large $m$ we prove that classical Smith theory holds for infinite-dimensional complexes where the role of   fixed points is played by   homotopy fixed points.  
 
\Thm{\label{A}\cite{AGapp} Let $P$ be a finite $p$--group   and $X\simeq (\sS^{n})^\wedge_p$  be a $P-$space.
Then there exists a positive integer $m$ such that $(X[m])^{hP}\simeq (\sS^r)^\wedge_p$ for some $r$.  
} 
  
We are informed that this result is going to appear in \cite{AGapp} and it is part of a program on homotopy group actions due to Jesper Grodal which was announced previously. Using this result we can define dimension functions for mod-$p$ spherical fibrations. 
A mod-$p$ spherical fibration is a fibration whose fiber has the homotopy type of a $p$--completed sphere.
Given a mod-$p$ spherical fibration $\xi:E\rightarrow BG$ and a $p$-subgroup $Q\leq G$, we can restrict the fibration $\xi$ to a fibration $\xi|_{BQ} : E_Q\to BQ$ by taking pull-back via the inclusion map $BQ\to BG$. This corresponds to restricting the $G$-action on $X_\xi$ to a $Q$-action via the inclusion map. We define the integer $n_{\xi[m]}(Q)$ via the weak equivalence

$$
(X_\xi[m])^{hQ}\simeq (\sS^{n_{\xi[m]}(Q)-1})^\wedge_p
$$
which is a consequence of Theorem \ref{A}.
It turns out that for $m$ large enough  $n_{\xi[m]}$ satisfies the Borel-Smith conditions when regarded as an integer valued function on the set of $p$--subgroups of $G$   (see Theorem \ref{BS}).
The dimension function can be made independent of $m$ by considering a rational valued dimension function defined as follows:
$$
\Dim_\xi(Q)= \frac{1}{m} n_{\xi[m]}(Q)
$$
for every $p$--subgroup $Q\leq G$.

The Euler class of a fibration is said to be $p$-effective if its restriction to elementary abelian $p$-subgroups of maximal rank is  non-nilpotent. This is a condition on the Euler class of a spherical fibration that is required to obtain a free action of a rank two group on a product of two spheres using the Adem-Smith method. As an application of the dimension function that we defined, we obtain the following.  
 
\Thm{\label{B}\cite{AGapp} Assume $p>2$. There exists \underline{no} mod-$p$ spherical fibration $\xi:E\rightarrow B\Qd(p)$ with a $p$--effective Euler class.}

We are informed that this result is also going to appear in \cite{AGapp} and it was previously announced as a theorem by Jesper Grodal.   As a consequence of Theorem \ref{B}, we obtain that the Adem-Smith method of constructing  free actions on  finite complexes homotopy equivalent to a product of spheres does not work for $\Qd(p)$.  

  Another method for constructing free actions on a product of two spheres $\sS^{n_1} \times \sS ^{n_2}$  is given by Hambleton and \" Unl\" u \cite{HU09}. This method applies only to the equal dimensional case ($n_1=n_2$). The following theorem shows that this method cannot be used for $\Qd(p)$ either. 

\Thm{\label{C} Let $G=\Qd(p)$. Then for any $n\geq 0$, there is \underline{no} finite free $G$-CW-complex $X$
homotopy equivalent to $\sS^n \times \sS^n$. }
 
Therefore if Benson-Carlson conjecture holds then in the construction of a complex $X\simeq \sS^{n_1}\times \sS^{n_2}$ with free $\Qd(p)$-action the possibilities are narrowed down to distinct dimensional spheres with a more exotic action. 

 The general theory of homotopy group actions has been considered by Adem and Grodal \cite{AGapp}. They  have informed us that Theorems \ref{A} and \ref{B} will also appear in their paper under preparation. The idea of using dimension functions for studying mod-$p$ spherical fibrations goes back to Grodal and Smith's unpublished earlier work, although an outline of their ideas can be found in the extended abstract \cite{GSrep}.  Theorems \ref{A} and \ref{B} can also be thought of as corollaries of a program on homotopy group actions due to Grodal. 
We are grateful to Adem and Grodal for sharing their ideas with us on the subject, and we are looking forward to reading their complete account on the subject. Here we offer our proofs for Theorems \ref{A} and \ref{B} for completeness and to cover a gap in the existing literature. We should also mention that a result stated by Assadi \cite[Corollary 4]{Ass88} also implies Theorems \ref{A} and \ref{B}. Unfortunately, no proofs were provided for this result either.

The organization of the paper is as follows. In Section \ref{Tfunc} we compute $\Fix(HE)$ for a fibration $\xi:E\rightarrow B\ZZ/p$ whose fiber has the cohomology of a sphere. Our main result Theorem \ref{A} (Theorem \ref{hofix}) is proved in Section \ref{cohfix} where we study   the space of sections of a mod-$p$ spherical fibration over the classifying space of a $p$--group. The dimension function for an $m$--fold join of a  mod-$p$ spherical fibration is defined in  Section \ref{dimfun}. We prove the non-existence result Theorem \ref{B} (Theorem \ref{nospherical}) in this section. In Section \ref{prodaction} we prove Theorem \ref{C} (Theorem \ref{thm:noprodaction}). We collected some results about mapping spaces, homotopy fixed points, and fiber joins in an  appendix  in Section \ref{appendix}.
 
 \vskip 5pt

{\bf Acknowledgement:} We thank Alejandro Adem, Matthew Gelvin, and Jesper Grodal for their comments on the first version of this paper. 
The second author is supported by T\" ubitak 1001 project (grant no: 116F194).

 \section{Spherical fibrations and Lannes' $T$-functor}\label{Tfunc} 
In this section we compute $\Fix(HE)$ for a mod-$p$ spherical fibration $E\rightarrow B\ZZ/p$. More generally, we work with fibrations where the mod-$p$ cohomology of the fiber is isomorphic to the cohomology of a sphere.
We modify the argument of \cite[Chapter 3, \S 4]{Die87} for the classical case which works for group actions on finite dimensional complexes and use the connection between Lannes' $T$-functor and localization, established in \cite{DW91}.
 
\subsection{Lannes $T$-functor:}  Let $\catU$ (resp. $\catK$) denote the category of unstable modules (resp. unstable algebras) over the mod-$p$ Steenrod algebra $\aA_p$. Let $V$ denote an elementary abelian $p$-group and $HV$ the mod-$p$ cohomology ring of $V$. The tensor product functor $HV\otimes-:\catU\rightarrow \catU$ has a left adjoint $T^V:\catU\rightarrow \catU$ which is called the Lannes $T$-functor.  
Let $\catU(HV)$ denote the category of unstable modules $M$ with an $HV$-module structure such that the  multiplication map $HV\otimes M\rightarrow M$ satisfies the Cartan formula.  Let $f:HW\rightarrow HV$ denote the map induced by a subgroup inclusion $V\subset W$. Its adjoint $\hat{f}:T^VHW\rightarrow \FF_p$  is determined by a ring homomorphism $\hat{f}_0:(T^VHW)^0\rightarrow \FF_p$ in degree zero. We define
$$T_f^V(M)=\FF_p\otimes_{(T^VHW)^0}T^VM $$
 where the $(T^VHW)^0$-module structure on $\FF_p$ is the one determined by $\hat f_0$.
Let $S_f$ denote the multiplicatively closed subset of $HV$ generated by the images of the Bocksteins of one-dimensional classes in $HW$ that map non-trivially under $f$. The  following is the main theorem of  \cite{DW91}.
\Thm{(Dwyer-Wilkerson \cite{DW91}) \label{loc_T}
Let $W$ be an elementary abelian $p$-group, $V$ a subgroup
of $W$, and $f:HW\rightarrow HV$ the map induced by subgroup inclusion.
Suppose that $M$ is an object of $\catU(HV)$ which is finitely-generated as a
module over $HV$. Then there is a natural isomorphism
$$
T_f^V(M)\cong \Un\; S^{-1}_f(M).
$$
}

For an object $M$ in $\catU(HV)$ the $\Fix$ functor is defined by  
$$
\Fix_V(M)=\FF_p\otimes_{T^VHV} T^VM 
$$
where $\FF_p$ is regarded as a $T^VHV$-module via the adjoint $\hat \varphi: T^VHV\rightarrow \FF_p$ of the identity map $\varphi:HV\rightarrow HV$, see \cite[\S 4.4.3]{Lan75} for details.  
We record the following properties.

\Pro{\label{T_pro} Let $M$ be an object in $\catU(HV)$.
\begin{enumerate}
\item The natural map $T_\varphi^V M \rightarrow HV \otimes \Fix_V M$ is an isomorphism in $\catU(HV)$.
\item If $M$ is a finitely generated $HV$-module then the localization of the natural map $M\rightarrow T^V_\varphi M$ with respect to $S_\varphi$ is an isomorphism.
\end{enumerate}
} 
\Proof{The first result is proved in \cite[Proposition 4.5]{Lan75}. In the second result the natural map is obtained as follows. Let $M\rightarrow HV\otimes T^VM$ denote the adjoint of the identity map $T^VM\rightarrow T^VM$. Composing this map with the unique algebra map $HV\rightarrow \FF_p$ gives a map $M\rightarrow T^VM$. The desired map is obtained by applying the natural projection $T^VM\rightarrow T^V_\varphi M$ to the second factor. The fact that the resulting map is an isomorphism can be found in \cite[Lemma 4.3, \S 5]{DW91}.}

\subsection{Spherical fibrations over $B\ZZ/p$} 
We will study fibrations  $\xi: E\rightarrow B\ZZ/p$ where the cohomology $HF$ of the fiber $F$ is isomorphic to $H(\sS^n)$ for some $n\geq 0$, and show that $\Fix(HE)\cong H(\sS^r)$ for some $-1\leq r\leq n$. Note that mod-$p$ spherical fibrations satisfy this condition.

We start with recalling the mod-$p$ cohomology ring of $\ZZ/p$.
If $p=2$ the cohomology ring
$H(\ZZ/2)$  is a polynomial algebra $\FF_2[t]$ where $t$ is of degree one. When $p>2$ we have $H(\ZZ/p)=\FF_p[t]\otimes\Lambda[s]$ where $s$ is of degree one and $t=\beta s$. Here $\beta$ is the Bockstein map.
The set $S_\varphi$ corresponding to the identity map $\varphi:H\ZZ/p\rightarrow H\ZZ/p$ is generated by the Bockstein of the one dimensional class in each case. If $S=\set{1,t,t^2,\cdots}$ then localization with respect to $S$ is the same as localization with respect to $S_\varphi$. For simplicity of notation, when $V=\ZZ/p$ we will write  $T=T^{V}$, $T_\varphi=T^{V}_\varphi$, and $\Fix=\Fix_{V}$.

\Lem{\label{empty} For an arbitrary fibration $\xi:E\rightarrow BV$ we have $(\Fix_V HE)^0=0$ if and only if $\xi^*:HV\rightarrow HE$ does not split in $\catK$.
}
\Proof{$(\Fix_VHE)^0$ is isomorphic to  $(T_\varphi^V HE)^0$ which  has a $\FF_p$-basis $Z_\varphi$ given by the set  of $\catK$-maps $\alpha:HE\rightarrow HV$ such that $\alpha\xi^*$ is the identity map on $HV$, see \cite[Theorem 3.8.6]{Sch94}.
}

Now we are ready to prove our main result in this section.

\Thm{\label{rankfix} Let $\xi:E\rightarrow B\ZZ/p$ be a fibration such that $HF\cong H(\sS^n)$. Then $\Fix\,(HE)\cong H(\sS^r)$ for some $-1\leq r\leq n$.  } 
\begin{proof} 
The Serre spectral sequence of the fibration $\xi$ has $E_2$-page given by
$$
H(\ZZ/p) \otimes HF \Rightarrow HE,
$$
which is non-zero only in two rows since $HF\cong H(\sS^n)$. The spectral sequence is determined by the differential $d_{n+1}:E_2^{0,n}\rightarrow E_2^{n+1,0}$ whose image lies in the polynomial part of $H\ZZ/p$ \cite[page 137]{AP93}.
First we assume that $d_{n+1}$ is non-zero i.e. $\xi^*$ does not split.   In this case
 $t$ is nilpotent in $HE$. Hence the localization vanishes:  $S^{-1}HE=0$. By Theorem \ref{loc_T} we have $T_\varphi HE =0$ and  the first part of Proposition \ref{T_pro} implies that $\Fix\, HE =0$.
Next assume that  $d_{n+1}=0$ i.e. $\xi^*$ splits. Localizing the natural map $HE\rightarrow T_\varphi HE$ with respect to $S$ gives a diagram
\begin{equation}\label{D1}
\begin{tikzcd}
HE \arrow{r} \arrow{d} & T_\varphi HE \arrow[d,hook] \\
S^{-1} HE \arrow[r,"\cong"] & S^{-1}(T_\varphi HE )
\end{tikzcd}
\end{equation}
Here the isomorphism is a consequence of the second property in Proposition \ref{T_pro}. The monomorphism maps onto the unstable part of $S^{-1}(T_\varphi HE )$ as a consequence of Theorem \ref{loc_T} and the commutativity of the diagram. 
 Since $HE\cong H(\ZZ/p)\otimes HF$   is a free $H(\ZZ/p) -$module generated by an element of degree $n$ the localization map  $HE\rightarrow S^{-1}HE$ is a monomorphism. Note that in the spectral sequence multiplication by $t$ is an isomorphism. After localizing the spectral sequence the two rows extend to negative degrees. 
Therefore comparing the spectral sequences we see that the localization map is an isomorphism in degrees $i\geq n$. 
Hence from the diagram \ref{D1} it follows that the natural map $H^iE\rightarrow (T_\varphi HE)^i\cong (H(\ZZ/p) \otimes \Fix(HE))^i$ is an isomorphism for $i\geq n$. Therefore we have
\begin{equation}\label{iso}
\oplus_{k=0}^i \Fix(HE)^k \cong \oplus_{k=0}^i H^k(F) = \FF_p\oplus \FF_p\;\;\;\text{ for } i\geq n.
\end{equation}
One of the factors corresponds to a generator of $\Fix(HE)^0$ which is non-zero by Lemma \ref{empty}. The other one corresponds to a generator of $\Fix(HE)$ in degree $r\leq n$. Therefore $\Fix(HE)$ is isomorphic to $H(\sS^r)$ where $0\leq r\leq n$.
\end{proof}

\section{Cohomology of homotopy fixed points}\label{cohfix}

In this section we study the homotopy fixed point space or equivalently the space of sections of a fibration by applying Lannes' results. For the relationship between the $T$--functor and mapping spaces, our main references are \cite{Lan75} and \cite{Sch94}. See also Appendix \S \ref{appendix} for preliminaries on homotopy fixed points and space of sections.

\subsection{Lannes' theorems}
 Lannes' $T$-functor gives an approximation to the cohomology of the mapping space $\Map(BV,Y)$ where $Y$ is an arbitrary space. The evaluation map $\ev:BV\times \Map(BV,Y)\rightarrow Y$ induces a map in cohomology $HY\rightarrow HV\otimes H(\Map(BV,Y))$ whose adjoint is $ T^VHY\rightarrow H(\Map(BV,Y))$. The adjoint map factors through 
\begin{equation}\label{evhat}
\hat{\ev}: T^VHY\rightarrow H(\Map(BV,Y^\wedge_p)).
\end{equation}
In degree zero it is induced by the isomorphism $[BV,Y^\wedge_p]\rightarrow \catK(HY,HV)$ defined by applying the cohomology functor \cite[pg. 187]{Sch94}. It is convenient to work at a connected component associated to  the homotopy class of a map $\alpha:BV\rightarrow Y^\wedge_p$. Let $\alpha^*:HY\rightarrow HV$ denote the homomorphism induced in cohomology. 
The component of $T^VHY$ at $\alpha^*$ is defined by
$$
T^V(HY,\alpha^*)=\FF_p\otimes_{(T^VHY)^0} T^VHY 
$$
where the module structure on $\FF_p$ is given by the adjoint  $T^VHY\rightarrow \FF_p$ of $\alpha^*$.   Then $\hat \ev$ in \ref{evhat} is the product of the maps
$$
\hat{\ev}_\alpha: T^V(HY,\alpha^*)\rightarrow H(\Map(BV,Y^\wedge_p)_\alpha)
$$
where $\alpha$ runs over the  homotopy classes of maps $BV\rightarrow Y^\wedge_p$.

We need the notion of freeness for the next theorem due to Lannes. Let $G$ denote the left adjoint of the forgetful functor $\catK\rightarrow \catE$ where  $\catE$ denotes the category of graded vector spaces over $\FF_p$. 
For an object $K\in \catK$ let $\Sigma K^1$ denote the graded vector space isomorphic to $K^1$ in degree one and zero in other degrees. There is an inclusion of graded vector spaces $\Sigma K^1 \rightarrow K$. Applying $G$ to this map and composing with the counit $GK\rightarrow K$ of the adjunction gives a
  canonical map 
\begin{equation}\label{canonical_free}
\chi:G(\Sigma K^1) \rightarrow K.
\end{equation} 
 An unstable algebra $K$ is said to be \textit{free in degrees $\leq 2$} if  $\chi$  is an isomorphism in degrees $< 2$  and a monomorphism in degree $2$.  For a more explicit definition see   \cite[pg. 25]{Lan75}.

\Thm{\label{free}(Lannes \cite[Theorem 3.2.4]{Lan75}) Assume that $HY$ and $T^VHY$ are of finite type. If $T^VHY$ is free in degrees $\leq 2$ then
$$
\hat{\ev}_\alpha: T^V(HY,\alpha^*)\rightarrow H(\Map(BV,Y^\wedge_p)_\alpha)
$$
is an isomorphism of unstable algebras.
} 

In the next section we will apply this theorem to mod-$p$ spherical fibrations.

\subsection{Mod-$p$ spherical fibrations}
A fibration whose fiber is homotopy equivalent to a $p$--completed sphere is called a mod-$p$ spherical fibration.
A source for such fibrations is the fiberwise completion of spherical fibrations.
Let $\xi:E\rightarrow BV$ be a mod-$p$ spherical fibration with connected fiber.
There is a map of fibrations
$$
\begin{tikzcd}
E \arrow{d}{\xi} \arrow{r} & E^\wedge_p \arrow{d}{\xi^\wedge_p} \\
BV \arrow{r} & BV^\wedge_p
\end{tikzcd}
$$
where the horizontal maps are weak equivalences. In particular $E$ is $p$--complete. Moreover, the diagram is a homotopy pull-back diagram. This implies that  there is a weak equivalence $\Sec(\xi)\rightarrow\Sec(\xi^\wedge_p)$ induced by the $p$-completion map. Therefore in applying Lannes' theory we can ignore the $p$--completions up to weak equivalence.
We are interested in the cohomology of the space of sections of  $\xi$. The space of sections $\Sec(\xi)$ is weakly equivalent to the homotopy fixed point space $X_\xi^{hV}$ where   $X_\xi$ is the $V$--space defined as the pull-back of $\xi$ along  the universal bundle $EV\rightarrow BV$.
The space of homotopy sections $\hSec(\xi )$ is isomorphic to $BV\times \Sec(\xi )$ as  simplicial sets (\S \ref{mapping_space}). We will use Lannes' theory to study the cohomology of space of homotopy sections. Consider the diagram
$$
\begin{tikzcd}
{[BV,E ]} \arrow{r}{\cong} \arrow{d} & \catK(HE,HV) \arrow{d} \\
{[BV,BV]} \arrow{r}{\cong} & \catK(HV,HV)
\end{tikzcd}
$$
induced by $\xi$, where the horizontal maps are bijections. Let $Z_\varphi$ denote the subset   of maps in 
   $\catK(HE,HV)$ which splits $\xi^*$ induced in cohomology. 
 The subset of maps in $[BV,E ]$ which splits $\xi $ up to homotopy is in one-to-one correspondence with $Z_\varphi$.  Then we have 
$$
T_\varphi ^V HE = \prod_{\alpha^*\in Z_\varphi}T^V(HE,\alpha^*).
$$
and the product of the evaluation maps $\hat \ev_\alpha$  gives a map
\begin{equation}\label{map1}
T^V_\varphi HE \rightarrow H(\hSec(\xi )).
\end{equation} 
By Theorem \ref{free} this map is an isomorphism of unstable algebras if $T^V _{\varphi} HE$ is free in degrees $\leq 2$. Note that the conditions that $HE$ and $T^VHE$ are of finite type are satisfied in this case because of the spectral sequence calculation and by Theorem \ref{rankfix}.

\Thm{\label{free_app} Let $\xi:E\rightarrow BV$ be a mod-$p$ spherical fibration and $X_\xi$ denote the pull-back of $\xi$ along the universal fibration $EV\rightarrow BV$.
Assume that $\Fix_V(HE)\cong H(\sS^r)$ for some $-1\leq r\leq n$. If  $r\not=1 $ then 
$$H(X_\xi^{hV}) \cong H(\sS^r).$$ 
} 
\Proof{   By Section \ref{homotopy_fix}, we have
$X_\xi ^{hV}\simeq\Sec(\xi).$ Since $\hSec( \xi)  \simeq BV \times \Sec(\xi)$, it is enough to show that the map
in \ref{map1} is an isomorphism. When $r=-1$ the result follows from Lemma \ref{empty}. 
For the cases $r=0$ and $r>1$  
we check the freeness condition. 
In \ref{canonical_free} it turns out that the object $G(\Sigma K^1) $ is isomorphic to $HW$ where $W$ is the $\FF_p$--dual of $K^1$.
Therefore $\chi$ is an isomorphism in degrees $\leq 2$ if and only if $H^2 W \rightarrow K^2$ is a monomorphism.
We claim that  $T_\varphi^V HE $ is free in degrees $\leq 2$ when $r=0$ and $r>1$.
If $r=0$ then the set $Z_\varphi$ contains two maps $\alpha_0$, $\alpha_1$ and $T^V_\varphi HE = T^V(HE,\alpha_0) \oplus T^V(HE,\alpha_1)$ where each component is isomorphic to $HV$. Note that for $K=HV$ the map $\chi$ is an isomorphism.  Hence the freeness condition is satisfied. For $r>1$ we have $T^V_\varphi HE = T^V(HE,\alpha)$ for a unique homotopy class of a map $\alpha$.  The freeness property holds since $(T_\varphi^V HE)^1 = H^1V$ and $H^2V\subset (T_\varphi^V HE)^2$.  
}
 
\Rem{\rm{
Note that in general  $T_\varphi^V HE$ is not free in degrees $\leq 2$ when $r=1$. Hence in this case we cannot apply Lannes' Theorem \ref{free} to calculate the cohomology of the homotopy fixed point space, see also \cite[Theorem 4.9.3]{Lan75}. 
}} 
We turn to another theorem of Lannes to study the homotopy type of the homotopy fixed point space.

\Thm{\label{pcomp}\cite[Corollary 3.4.3]{Lan75}
Assume that $HY$, $T^V HY$, and $H(\Map(BV,Y)_\alpha)$ are of finite type. Then $  T^V (HY,\alpha^*) \rightarrow H(\Map(BV,Y)_\alpha)$ is an isomorphism of unstable algebras if and only if $(\Map(BV,Y)_\alpha )^\wedge_p \rightarrow \Map(BV,Y^\wedge_p)_\alpha$ is a homotopy equivalence.
}

Using this theorem, the Fix calculation Theorem \ref{rankfix}, and Theorem \ref{free_app} we can determine the homotopy type of the space of sections of a mod-$p$  spherical fibration  over $B\ZZ/p$.  

\Thm{\label{sec_ho_type} Let $\xi:E\rightarrow B\ZZ/p$ be a mod-$p$ spherical fibration such that $\Fix(HE)\cong H(\sS^r)$ where $r\not= 1$. Let $X_\xi$ denote the pull-back of $\xi $ along $E\ZZ/p\rightarrow B\ZZ/p$. Then 
$$X_\xi^{h\ZZ/p}\simeq  (\sS^r)^\wedge_p$$
where $-1\leq r\leq n$. 
} 
\Proof{In Theorem \ref{rankfix} we showed that $\Fix(HE)\cong H(\sS^r)$ for some $-1\leq r\leq n$. If $r\not=1$ then  Theorem \ref{free_app} implies that $H(\Sec(\xi))\cong H(\sS^r)$. The section space is the product of mapping spaces $\Map(BV,E)_\alpha$ where $\alpha$ is a representative of a homotopy class such that $\alpha^*$ lies in $Z_\varphi$.
Applying Theorem \ref{pcomp} to each component we obtain a homotopy equivalence
$$
\Sec(\xi )^\wedge_p \rightarrow \Sec(\xi ) 
$$
after identifying $\Sec(\xi )\simeq \Sec(\xi^\wedge_p)$ up to homotopy.
 Therefore $\Sec(\xi)\simeq X_\xi^{h\ZZ/p}$ is a $p$--complete space which has the cohomology of a sphere.
}

\subsection{Fiber joins and the $Fix$ functor} Next we look at the relationship between the $\Fix$ functor and fiber joins to be able to go around the problem  in Theorem \ref{sec_ho_type} when $r=1$.
Let us take two fibrations $\xi_1:E_1\rightarrow B\ZZ/p$ and $\xi_2:E_2\rightarrow B\ZZ/p$ with fibers denoted by $F_1$ and $F_2$, respectively. 
We assume that $HF_i\cong H(\sS^{n_i})$ and $\xi_1^*:H\ZZ/p\rightarrow HE_1$ splits.

\Lem{\label{fix_iso}  
The natural map 
$$
\Fix\,(HE_1)\otimes \Fix\,(HE_2) \rightarrow \Fix\, H(E_1\times_{B\ZZ/p}E_2)
$$
is an isomorphism.
}
\begin{proof}Consider the pull-back diagram of fibrations
$$
\begin{tikzcd}
&F_1 \arrow[r,equal] \arrow{d} & F_1 \arrow{d} \\
F_2\arrow{r} \arrow[d,equal] &E_1\times_{B\ZZ/p} E_2 \arrow{r}{p_1} \arrow{d}{p_2} & E_1\arrow{d}{\xi_1} \\
F_2 \arrow{r} &E_2 \arrow{r}{\xi_2}  & B\ZZ/p
\end{tikzcd}
$$
We have $HE_1=H\ZZ/p\otimes HF_1$. The differential $d_{n+1}$ in the spectral sequence of $\xi_2$ is either zero or $t^\alpha$.
By comparing the spectral sequences  we see that $H(E_1\times_{B\ZZ/p} E_2)$ is either isomorphic to $HE_1\otimes HF_2$ or $HE_1/(t^\alpha)$.
Consider the natural map
$$
\theta: HE_1 \otimes_{H\ZZ/p} HE_2 \rightarrow H(E_1\times_{B\ZZ/p} E_2)
$$
of unstable modules induced by $p_1$ and $p_2$. If $d_{n+1}=0$ then the tensor product is isomorphic to $HE_1 \otimes HF_2$. If the differential is given by $t^\alpha$ then it becomes $HE_1 \otimes_{H\ZZ/p} (H\ZZ/p)/(t^\alpha)\cong HE_1/(t^\alpha)$. Therefore in both cases  $\theta$ is an isomorphism of unstable modules.
 In fact it is a morphism in $\catU(H\ZZ/p)$. Then the result follows from the isomorphism $\Fix(M_1\otimes_{H\ZZ/p}M_2)\cong \Fix(M_1)\otimes \Fix(M_2)$ which is valid for $\catU(H\ZZ/p)-$modules  \cite[Theorem 4.6.2.1]{Lan75}.
\end{proof}

\Pro{\label{fix_join} 
Assume that   $\Fix\, H(E_i)\cong H(\sS^{r_i})$ for some $r_i$.
 Then there is an isomorphism
$$
\Fix\, (H(E_1\ast_{B\ZZ/p} E_2)) \cong H (\sS^{r_1+r_2+1}).
$$ 
}
\begin{proof}
Consider the homotopy push-out square
\begin{equation}\label{hpushout}
\begin{tikzcd}
 E_1\times_{B\ZZ/p} E_2 \arrow{r}\arrow{d} &  E_1 \arrow{d}\\
 E_2 \arrow{r} &  E_1\ast_{B\ZZ/p} E_2
\end{tikzcd}
\end{equation}
We can assume that $E_i=(X_i)_{h\ZZ/p}$  for some $\ZZ/p-$space $X_i$. The assignment $X\mapsto \Fix\, H(X_{h\ZZ/p})$ defines an equivariant cohomology theory on the category of $\ZZ/p-$spaces \cite[\S 4.7]{Lan75}.   Then associated to the push out diagram there is a Mayer-Vietoris sequence which   breaks into short exact sequences
$$
0 \rightarrow \Fix\, H(E_1)^q \oplus \Fix\, H(E_2)^q  \rightarrow  (\Fix\, H(E_1) \otimes \Fix \,H(E_2))^q  \rightarrow \Fix  \,H (E_1\ast_{B\ZZ/p} E_2) ^{q+1}\rightarrow 0
$$
where we used Lemma \ref{fix_iso} for the middle term.
In degree zero we need to consider the reduced groups. Now compare this sequence to the Mayer-Vietoris sequence of the homotopy push-out
$$
\begin{tikzcd}
\sS^{r_1}\times \sS^{r_2} \arrow{r}\arrow{d} &   \sS^{r_1} \arrow{d}\\
\sS^{r_2}  \arrow{r} & \sS^{r_1}\ast \sS^{r_2}
\end{tikzcd}
$$
The result follows from $\sS^{r_1}\ast \sS^{r_2}\cong \sS^{r_1+r_2+1}$.  
\end{proof}

Let $\xi:E\rightarrow B$ be a fibration with fiber $F$. The fiberwise $p$--completion (\S \ref{pre_completion}) of $\xi$ is a fibration $\xi^\wedge_{p/B}:E^\wedge_{p/B}\rightarrow B$ whose fiber is $F^\wedge_p$.  
We use the following notation (\S \ref{pre_fibjoin})
$$
X[m]= (\underbrace{X\ast\cdots \ast X}_m)^\wedge_p
$$
and for a fibration $\xi:E\rightarrow B$ we define 
$$
E_{/B}[m]=(\underbrace{E \ast_B\cdots \ast_B E}_m)^\wedge_{p/B}
$$ 
and denote the associated fibration by $\xi[m]:E_{/B}[m]\rightarrow B$.  

\Cor{\label{join_hfix}  
Let $X=X_\xi$ and $r$ be as defined in  Theorem \ref{sec_ho_type}. Then for all $m>2$ we have
$$
(X[m])^{h\ZZ/p} \simeq \sS^r[m].
$$
} 
\Proof{By Corollary \ref{not_equiv} we have a fiber homotopy equivalence $(X[m])_{h\ZZ/p}\simeq E_{/B\ZZ/p}[m]$. Since $X[m]$ is homotopy equivalent to  a $p$--completed sphere, using Proposition \ref{fix_join} we obtain
$$
 \Fix\, H(E_{/B\ZZ/p}[m]) \cong \Fix\, H(E\ast_{B\ZZ/p}\cdots \ast_{B\ZZ/p}E)\cong H(\sS^r[m]\,).
$$ Therefore we can apply Theorem \ref{sec_ho_type} to $\xi[m]$.
} 

\Rem{\label{rem_one}
\rm{ According to Theorem \ref{sec_ho_type}, as long as $r\not= 1$ the statement of Corollary \ref{join_hfix} holds with $m=1$. 
The problem we faced for $r=1$ can be handled by taking joins. When $r=1$ it suffices to take  $m=2$ to obtain
$$
(X[2])^{h\ZZ/p}   \simeq  (\sS^{3})^\wedge_p.
$$  
}} 
  
\subsection{Finite $p$--groups}  Next we  extend our results     to $p$--groups. 
Let $P$ be a finite $p$--group and $Z\cong \ZZ/p$ be a subgroup of $P$ contained in the center.
Consider a mod-$p$ spherical fibration $\xi:E\rightarrow BP$, and let $X=X_\xi$.
We are interested in computing the homotopy type of the homotopy fixed point space $X^{hP}$. By transitivity of homotopy fixed points (\S \ref{homotopy_fix}) we have
$$
X^{hP} \simeq Y^{hP/Z}
$$
where $Y=\Map(EP, X)^Z \simeq X^{hZ}$. By replacing  $X$ with $X[k]$ for some $k$ and using Corollary \ref{join_hfix}, we can ensure that $X^{hZ}$ is homotopy equivalent to a $p$--completed sphere. Now we can consider the $P/Z$--space $Y$.
But to be able to determine the homotopy type of $Y^{hP/Z}$ we may need to replace $Y$ with $Y[l]$. At this step we need the following lemma.
 
 \begin{lem}\label{join_weak} Let $X_1$ and $X_2$ be $P-$spaces such that for $i=1,2$, $X_i\simeq (\sS^{n_i})^\wedge_p$, and $X_i^{hZ}\simeq (\sS^{r_i})^\wedge_p$ for some $r_i>0$.
There is a  weak equivalence
$$
\alpha:(X_1^{hZ}\ast X_2^{hZ})^\wedge_p \rightarrow ((X_1\ast X_2)^\wedge_p)^{hZ}
$$ 
which is induced by a map of $P/Z$--spaces when the homotopy fixed point spaces are interpreted as mapping spaces.
\end{lem}
\begin{proof} We describe the map $\alpha$. 
There is a natural map of $P/Z-$spaces
$$
\alpha':\Map(EP,X_1)^Z\ast\Map(EP,X_2)^Z \rightarrow \Map(EP,X_1\ast X_2)^Z
$$ 
defined by $\alpha'[f,g,t](z)=[f(z),g(z),t]$. Note that $\Map(EP,X)^Z$ is weakly equivalent to $X^{hZ}$ via the natural map $EZ\rightarrow EP$. Hence we obtain 
$$X_1^{hZ}\ast X_2^{hZ}  \rightarrow (X_1\ast X_2)^{hZ}.$$
Composing this with the natural map $(X_1\ast X_2)^{hZ} \rightarrow ((X_1\ast X_2)^\wedge_p)^{hZ}$, we obtain a map
$$X_1^{hZ}\ast X_2^{hZ}  \rightarrow   ((X_1\ast X_2)^\wedge_p)^{hZ}.$$
Completion of this map at $p$ gives the map $\alpha$. Note that since $r_1+r_2+1>1$ we can apply Theorem \ref{sec_ho_type} to conclude that  $((X_1\ast X_2)^\wedge_p)^{hZ}$ is $p-$complete. To see that $\alpha$ is a weak equivalence it suffices to show that the map induced in mod-$p$ cohomology is an isomorphism. A Mayer-Vietoris type of argument shows that
$$
H((X_1^{hZ}\ast X_2^{hZ})^\wedge_p) \cong H(\sS^{r_1+r_2+1}).
$$
On the other hand Proposition \ref{fix_join} implies that $\Fix\,H(X_1\ast X_2)_{hZ}\cong H(\sS^{r_1+r_2+1})$ and by  Theorem \ref{sec_ho_type}  $((X_1\ast X_2)^\wedge_p)^{hZ}$ is weakly equivalent to  the $p-$completion of $\sS^{r_1+r_2+1}$.
\end{proof}

Now we are ready to prove the main theorem of this section.

\Thm{\label{hofix} Let $P$ be a finite $p$--group and $\xi:E\rightarrow BP$ be a mod-$p$ spherical fibration.
Then there exists a positive integer $m$ such that $X[m]^{hP}\simeq (\sS^r)^\wedge_p$ where $X=X_\xi$.
}
\Proof{  
We will use the transitivity property (\S \ref{homotopy_fix}) of homotopy fixed points:
$$
(\Map(EP, X)^Z)^{hP/Z} \simeq X^{hP}
$$
where $\Map(EP, X)^Z\simeq X^{hZ}$, and we will do induction on the order of $P$. 
We can assume $X^{hZ}$ is non-empty, otherwise the result holds trivially. 
Using Corollary \ref{join_hfix} we can replace $X$ by  $X[k]$ to ensure that  $Y=\Map(EP,X[k])^{Z}\simeq X[k]^{hZ}$ has the homotopy type of a $p$-completed simply connected sphere. We regard $Y$ as a $P/Z-$space. Since the order of $P/Z$ is less than the order of $P$  there exists,  by the induction hypothesis, some $l$ such that the homotopy fixed points $(Y[l])^{hP/Z}$ is weakly equivalent to a $p$--completed sphere. We claim that the homotopy fixed points of $X[kl]$   under the action of $P$ is a $p-$completed sphere. To see this let $A=X[k]$.  There is a weak equivalence
$$
(A^{hZ})[l] \rightarrow (A[l])^{hZ}
$$
which is induced by a map of $P/Z-$spaces when regarded as a map between the associated mapping spaces.  This can be shown by using Lemma \ref{join_weak} and doing induction on $l$. Now consider the natural $P$--map
$$
 X[kl] \rightarrow (X[k])[l] 
$$
which is also a weak equivalence. Using these two maps we obtain a zig-zag of weak equivalences
$$
(A^{hZ})[l] \rightarrow (A[l])^{hZ}\leftarrow   X[kl]^{hZ}
$$
through $P/Z-$maps. Thus we have 
$$X[kl]^{hP}  \simeq (Y[l])^{hP/Z}$$
and the result follows by induction.
}

\section{Dimension functions}\label{dimfun}
In this section we will define dimension functions for spherical fibrations and show that they satisfy the Borel-Smith conditions after taking fiber joins. 
 
\subsection{Dimension functions}  
Let $C(G)$ denote the ring of integer valued functions defined on the set of all subgroups of $G$ which are constant on $G$--conjugacy classes. Let $H\subset G$ be a subgroup. 
Given a finite $G$-CW-complex $Y$ with $H(Y)\cong H(\sS^n)$ the dimension function $n_Y$ of $Y$ is defined by 
 $H(Y^H)\cong H(\sS^{n_Y(H)-1})$.
Thus $Y$ gives rise to an element $n_Y$ of $C(G)$.

We extend this definition to mod-$p$ spherical fibrations.
Let $\xi:E\rightarrow BG$ be a mod-$p$ spherical fibration.
For a $p$--subgroup $Q\leq G$ let $X_Q$ denote the pull-back  $\xi|_{BQ}$ along the universal fibration $EQ\rightarrow BQ$.
By Theorem \ref{hofix} there exists an $m$ such that $(X_Q[m])^{hQ}$ has the homotopy type of $(\sS^{r_Q})^\wedge_p$ for some $r_Q$ for all $p$--subgroups $Q\leq G$. Note that by standard properties of homotopy fixed points and Theorem \ref{sec_ho_type}, we have $r_{Q'}\leq r_{Q}$ if $Q\leq Q'$ and $r_Q=r_{Q'}$ if $Q$ is conjugate to $Q'$ in $G$. 
  Let $\Ss_p(G)$ denote the set of all $p$--subgroups of $G$. 
 We define a  function 
 $$
 n_{\xi[m]}:\Ss_p(G) \rightarrow \ZZ\;\;\text{ by }\;\; n_{\xi[m]}(Q)=r_Q+1
 $$ 
which is constant on $G$--conjugacy classes and call it the dimension function associated to the fibration $\xi[m]$.
Given a spherical fibration we can consider the dimension function associated to its fiberwise $p$--completion.  
 
\Rem{\rm{
To associate a dimension function to a mod-$p$ spherical fibration independent of $m$ we can define a rational valued dimension function 
$$
\Dim_\xi(Q)=\frac{1}{m} n_{\xi[m]}(Q)
$$
for all $p$--subgroups $Q\leq G$ where $m$ is a positive integer large enough so that Theorem \ref{hofix} holds.
}}

\subsection{Dimensions and subgroups}
We will  prove an important relation satisfied by the dimension functions. Let $V$ be an elementary abelian $p$--group of rank two. Let $\xi:E\rightarrow BV$ be a mod-$p$ spherical fibration and $X=X_\xi$. Assume that $n_\xi$ is defined. (This can be achieved by replacing $\xi$ with $\xi[m]$.) This means that the homotopy fixed points of $X$ under the action of a subgroup of $V$ is a $p$--completed sphere.
 
By the Thom isomorphism theorem for $W\leq V$ the reduced cohomology ring of the Thom space $\Th(\xi_W)$ of the fibration $\xi_W: (X^{hW})_{hV}\rightarrow BV$ is a free $HV$-module on a single generator $t(\xi_W)$.
There is a map $X^{hV}\rightarrow X^{hW}$ defined as the composition
$$
\Map(EV,X)^V \rightarrow \Map(EV,X)^W \rightarrow \Map(EW,X)^W
$$
of the natural inclusion of the fixed points, and the map induced by $EW\rightarrow EV$.
This map induces a diagram  
\begin{equation}\label{cofib_diag}
\begin{tikzcd}
(X^{hV})_{hV} \arrow{d}{\xi_{V}} \arrow{r} & (X^{hW})_{hV} \arrow{d}{\xi_{W}} \\
BV \arrow[r,equal] \arrow{d} & BV\arrow{d} \\
\Th(\xi_{V}) \arrow{r}{\alpha} & \Th(\xi_{W}) \arrow{r}{\beta} & \Sigma (X_{W,V})_{hV}
\end{tikzcd}
\end{equation}
where $X_{W,V}$ is the cofiber of   $X^{hV}\rightarrow X^{hW}$. 
The bottom row is a cofibration sequence. 
In the  long exact sequence of cohomology groups 
\begin{equation}\label{longexact}
\cdots\rightarrow \tilde H^{i}(\Th(\xi_{W})) \stackrel{\alpha^*}{\rightarrow}\tilde H^{i}(\Th(\xi_V))  \stackrel{\beta^*}{\rightarrow}  H^{i}(\,(X_{W,V})_{hV}) \rightarrow \cdots
\end{equation}
 we have $\alpha^*(t(\xi_{W}))=e_{W,V}t(\xi_{V})$ for some element $e_{W,V}$ in $HV$.   
Let us set $S_W=S_{f}$ where $f:HV\rightarrow HW$ is the map induced by a subgroup inclusion $W\subset V$.

\Lem{\label{coh_cofib} Let $W\subset V$ be a subspace of codimension one. Then there is an isomorphism $ H(\,(X_{W,V})_{hV})\cong HV/(e_{W,V})$.} 
\begin{proof}
Let $Y$ denote the space of homotopy fixed points $\Map(EV,X)^W\simeq X^{hW}$. Let  $V=W\times L$ be a splitting. Then $Y^{hL}\simeq X^{hV}$ and $(X^{hV})_{hV} \rightarrow (X^{hW})_{hV}$ induce a map in cohomology $H((X^{hW})_{hV})\rightarrow H((X^{hV})_{hV})$. Note that $(X^{hW})_{hV}\simeq BW\times Y_{hL}$ and $(X^{hV})_{hV}\simeq BW\times (Y^{hL})_{hL}$. Therefore we obtain
$$
HW \otimes H(Y_{hL}) \rightarrow HW \otimes H((Y^{hL})_{hL})
$$
which   becomes an isomorphism after localizing with respect to $S_L$. This is a consequence of the isomorphism $T^L_\varphi H(Y_{hL})\cong H((Y^{hL})_{hL})$ implied by Theorem \ref{free_app} and the second part of Proposition \ref{T_pro} applied to $M=H(Y_{hL})$.
Therefore localization of $H((X^{hW})_{hV})\rightarrow H((X^{hV})_{hV} )$ with respect to $S_V$ is an isomorphism. From the map between the cofiber sequences in \ref{cofib_diag} we see that  the map between the cohomology rings of Thom spaces becomes an isomorphism after localizing with respect to $S_V$. Thus there is a diagram
$$
\begin{tikzcd}
\tilde H(\Th(\xi_W)) \arrow{r}{\alpha^*} \arrow{d} & \tilde H(\Th(\xi_V)) \arrow{d}\\
S^{-1}_V\tilde H(\Th(\xi_W)) \arrow{r}{\cong} & S^{-1}_V \tilde H(\Th(\xi_V)) 
\end{tikzcd}
$$
where the vertical arrows are injective since $\tilde H(\Th(\xi_W))$ and $\tilde H(\Th(\xi_V))$ are $HV$-free. 
Therefore  in \ref{longexact} we have that $\alpha^*$ is injective and $\beta^*$ is surjective. Then $  H(\,(X_{W,V})_{hV})$ is the quotient of the map
$$
\tilde H(\Th(\xi_W)) \rightarrow \tilde H(\Th(\xi_V))  
$$
induced by $t(\xi_W)\mapsto e_{W,V}t(\xi_V)$. 
\end{proof}

  Let us simply denote $e_{W,V}$ by $e_V$ when $W$ is the trivial group.  
 Let $t_L$ denote the generator of the polynomial part of $HL$. We regard $t_L$ as an element of $HV$ via the isomorphism $HV\cong HL\otimes HW$.

\Lem{\label{lem1}  
Assume that $e_V$ belongs to the polynomial part of $HV$. We have $n_X(W)>n_X(V)$ if and only if $t_L$ divides $e_V$. Moreover 
$$e_V= u\prod_W e_{W,V}$$
where $u\in \FF_p$ is a unit and $W$ runs over subspaces of codimension one in $V$ such that $n_X(W)>n_X(V)$.
} 
\begin{proof}
Note that $n_X(W)>n_X(V)$ if and only if $e_{W,V}=a t_L^\alpha$ for some $\alpha>0$ and $a\in \FF_p$ is non-zero. Let $\beta$ be the maximal natural number such that $t_L^\beta$ divides $e_V$. 
Consider the last cofiber sequences in Diagram \ref{cofib_diag} for the pair of subgroup inclusions  given by $W\subset V$ and $1\subset V$. There is a map between the cofiber sequences
$$
\begin{tikzcd}
\Th(\xi_V ) \arrow{r} \arrow[d,equal] & \Th(\xi_W) \arrow{r} \arrow{d} & \Sigma\, (X_{W,V})_{hV} \arrow{d} \\
\Th(\xi_V ) \arrow{r}  & \Th(\xi_1) \arrow{r}   & \Sigma\, (X_{1,V})_{hV} \ .
\end{tikzcd}
$$
We claim that the map
$S_W^{-1}H(X_{hV})\rightarrow S_{W}^{-1}H((X^{hW})_{hV})$ is an isomorphism. This   follows from 
the transitivity of the Borel construction.
The map $((X^{hW})_{hW})_{hL}\rightarrow (X_{hW})_{hL}$ between the Borel constructions with respect to the action of $L$ induces a map between the $E_2$-pages
$$
H(L, S_W^{-1}H(X_{hW})) \rightarrow H(L, S_W^{-1}H((X^{hW})_{hW}))
$$
of the associated spectral sequences. Then the claim follows from the isomorphism  
$$S^{-1}_WH(X_{hW}) \cong S^{-1}_WH((\,X^{hW}) _{hW}).$$  Now comparing the diagrams for $W\subset V$ and $1\subset V$ we see that localization of $H((X_{1,V})_{hV} )\rightarrow H((X_{W,V})_{hV} )$ with respect to $S_W$ is an isomorphism. Thus by Lemma \ref{coh_cofib} there is an isomorphism $S_W^{-1} HV/(e_V)\cong S_W^{-1} HV/(e_{W,V})$. This forces $\alpha=\beta$.

\end{proof}

\Pro{\label{sum}  
Assume $e_V$ belongs to the polynomial part of $HV$. If $W_1,\cdots,W_s$ denote the subspaces  of codimension one in $V$ then
$$
n_X(1)-n_X(V)=\sum_{i=1}^s (n_X(W_i)-n_X(V)).
$$
}
\Proof{ By Lemma \ref{lem1} a codimension one subspace $W$ contributes to the sum on the right-hand side if and only if $t_L$ divides $e_V$. Therefore the result follows from comparing the degrees. Note that
$
|e_V|=n_X(1)-n_X(V)  
$
and $|e_{W,V}|=n_X(W)-n_X(V) $.
}
 
\Rem{\rm{
In view of Corollary \ref{prod_Euler}, by choosing $m$ large enough in $\xi[m]$ we can achieve the property that $e_V$ belongs to the polynomial part of $HV$.
}} 

\subsection{Borel-Smith functions}  
An element $\tau$ of $C(G)$ is called a Borel-Smith function (\cite[Definition 5.1]{Die87}) if it satisfies the following conditions 
\begin{enumerate}[(i)]
\item if $H\normal K \subset G$ such that $K/H\cong \ZZ/p\times \ZZ/p$ and $H_i/H$ denotes the cyclic subgroups then
$$
\tau(H)-\tau(K)=\sum_{i=0}^p (\tau(H_i)-\tau(K)),
$$
\item if $H\normal K \subset G$ such that $K/H\cong \ZZ/p$ where $p>2$ then $\tau(H)-\tau(K)$ is even, and
\item if $H\normal L \normal K \subset G$ such that $L/H\cong \ZZ/2$ then $\tau(H)-\tau(L)$ is even if $K/H\cong \ZZ/4$; $\tau(H)-\tau(L)$ is divisible by $4$ if $K/H$ is a generalized quaternion group of order $\geq 2^3$.
\end{enumerate} 
Let $C_{b}(G)$ denote the additive subgroup  of Borel-Smith functions in $C(G)$. We also say a function $\Ss_p(G)\rightarrow \ZZ$ constant on the $G$--conjugacy classes satisfies Borel-Smith conditions if it satisfies (i), (ii), and (iii) on $p$--subgroups.

\Thm{\label{BS} There exists a positive integer $m$ such that $n_{\xi[m]}:\Ss_p(G)\rightarrow \ZZ$ satisfies the Borel-Smith conditions.}
\Proof{  Monotonicity is a consequence of Theorem \ref{sec_ho_type}.
Condition (i) is proved in Proposition \ref{sum}, where the hypothesis that $e_V$ belongs to the polynomial part of $HV$ holds by choosing $m$ large enough. This is a consequence of Corollary \ref{prod_Euler}. The  conditions (ii) and (iii)  can be achieved by taking $m$ large.
}

When $G$ is a finite nilpotent group
Borel-Smith functions   can be realized as dimension functions of virtual representations.
Let $RO(G)$ denote the Grothendieck group of real representations.
There is an additive morphism $\dim:RO(G)\rightarrow C(G)$ which sends a real representation $\rho$ to the function which sends a subgroup $H$ to the dimension of the fixed subspace $\rho^H$. Let $C_\rep(G)$ denote the image of this homomorphism. A key fact is that if $G$ is a finite nilpotent group then $C_b(G)=C_\rep(G)$. In the case of $p$-groups we can use honest representations when the Borel-Smith function is also monotone.

\Thm{\cite[Theorem 5.13]{Die87} \label{rep}  If $P$ is a $p$-group and $\tau$ is a monotone Borel-Smith function then there is a real representation $\rho$ such that $\tau=\dim \rho$.}

Up to fiber joins the dimension function of a mod-$p$ spherical fibration can be realized by the dimension function of a real representation.

\Cor{Let $P$ be a finite $p$-group. Given a  mod-$p$ spherical fibration $\xi:E\rightarrow BP$  there is a positive integer $m$ and a real representation $\rho$ such that $n_{\xi[m]}=\dim\rho$.}

\subsection{Proof of Theorem \ref{B}} 
A cohomology class in $H(G)$ is called \textit{$p$-effective} if its restriction to maximal elementary abelian $p$-subgroups is non-nilpotent. Let $\Qd(p)$ denote the semi-direct product $(\ZZ/p)^2  \rtimes \SL_2(\ZZ/p) $ where the special linear group acts in the obvious way.

\Cor{\label{nospherical} Assume $p>2$. There exists \underline{no} mod-$p$ spherical fibration $\xi:E\rightarrow B\Qd(p)$ with a $p$-effective Euler class.}
\begin{proof} The idea of the proof follows \cite[Theorem 3.3]{Unl04} but we use the dimension functions for  mod-$p$  spherical fibrations.
Assume that there is  a fibration $\xi$ with an effective Euler class. Consider the dimension function $n_{\xi[m]}$ for some large $m$. Its Euler class is still effective by Corollary \ref{prod_Euler}.
Let $P$ be a Sylow $p$--subgroup of $G=\Qd(p)$. The center $Z(P)$ is a cyclic group of order $p$. 
 By Theorem \ref{BS} we can choose $m$ large enough so that the dimension function of the restricted bundle $\xi[m]|_{BP}$ belongs to $C_{b}(P)$. Then Theorem \ref{rep} implies that there is a real representation $\rho$ which realizes this dimension function. Since the Euler class is $p$--effective the  dimensions of  subspaces of $\rho$ fixed under a cyclic subgroup $C\leq P_i$ has the property that $\dim\rho^{C}= 0$ if and only if $C=Z(P)$ \cite[Lemma 3.4]{Unl04}. Therefore on cyclic $p$--subgroups of $P$ the dimension function $n_{\xi[m]}$ is zero only at $Z(P)$ but in $G$ the center $Z(P)$ is conjugate to a non-central cyclic $p$--subgroup $C$ of $P$. Then $n_{\xi[m]}(Z(P))=n_{\xi[m]}(C)$ but this gives a contradiction.
\end{proof}

\section{$Qd(p)$-action on $\sS^n \times \sS^n$ }\label{prodaction}

In this section we prove the following theorem.

\Thm{\label{thm:noprodaction} Let $G=\Qd(p)$. Then for any $n \geq 0$, there is \underline{no} finite free $G$-CW-complex $X$
homotopy equivalent to $\sS^n \times \sS^n$. }

If $p=2$, then  $G=\Qd(2)$ is isomorphic to the symmetric group $S_4$ that includes $A_4$ as a subgroup. In this case the theorem 
follows from a result of Oliver \cite{Oli79} which says that the group $A_4$ does not act freely on a finite complex $X$ homotopy equivalent to a product of two equal dimensional spheres. Also note that for $n=0$, the statement holds for obvious reasons. Hence 
it is enough to prove the theorem when $p$ is an odd prime and $n \geq 1$.

\Lem{Let $G$ be a finite group generated by elements of odd order. Let
$X$ be a finite free $G$-CW-complex homotopy  equivalent to $\sS^n\times \sS^n$ for some $n \geq 1$. 
Then $n$ is odd, and the induced $G$-action on $H^* (X; \bZ) $ is trivial. 
}

\Proof{ It is enough to prove this for the case $G=\bZ/p^k$, where $p$ is an odd prime.   By induction we can assume 
that the action of the maximal subgroup $H \leq G$  on cohomology is trivial. Consider the $G/H \cong \bZ/p$ action on $H^* (X; \bZ)$.
The only indecomposable $\bZ$-free $\bZ[\bZ/p]$-modules are either $1$-dimensional, $(p-1)$-dimensional, or $p$-dimensional \cite[Theorem 2.6]{HR62}.  
This gives that for $p>3$, the $G$ action on $H^* (X; \bZ)$ is trivial. For $p=3$, the only nontrivial module can occur in
dimension $n$, and in this case $G/H$ acts on $H^n (X; \bZ)$ with the action $x\to -y$ and $y\to x-y$, where $x, y$ are generators
of $H^n (X; \bZ)\cong \bZ \oplus \bZ$. Note that the trace of this action is $-1$, so by the Lefschetz trace formula $L(f)=2-(-1)=3$ when $n$ is odd, and $L(f)=2+(-1)=1$ when $n$ is even. In either case $L(f)\neq 0$, hence $G$ cannot admit a free action on $X$ if the $G/H$ action on homology is nontrivial. If the action is trivial, then again by Lefschetz trace formula, $n$ must be odd.
}

The group $SL_2(p)$ is generated by elements of order $p$. For example, we can take
$$A=\{ \begin{bmatrix} 1 & 1 \\ 0 & 1 \end{bmatrix}, \begin{bmatrix}  1 & 0 \\ 1 & 1 \end{bmatrix} \} $$ as a set of generators. 
Since $\Qd(p)= (\ZZ/p)^2\rtimes\SL_2(\ZZ/p)$ is a semidirect product of $(\bZ/p )^2 $ with $SL_2 (p)$, it is also generated by elements of order $p$. Hence we conclude the following.

\Pro{\label{pro:trivial}  Let $G=\Qd(p)$ where $p$ is an odd prime.
Suppose that there exists a finite free $G$-CW-complex $X$
homotopy equivalent to $\sS^n \times \sS^n$ for some $n \geq 1$. Then $n$ is odd
and $G$ acts trivially on $H^* (X; \bZ)$.  
}

To complete the proof of Theorem \ref{thm:noprodaction}, we use the Borel construction.  Let 
$G=\Qd(p)$ with $p$ odd, and  let $X$ be a finite free $G$-CW-complex
homotopy equivalent to $\sS^n \times \sS^n$ for some integer $n\geq 1$. By Proposition \ref{pro:trivial},  the induced action 
of $G$ on $X$ is trivial and $n=2k-1$ for some $k\geq 1$.  
Consider the Borel fibration $X_{hG}\rightarrow BG$ where $X_{hG}=EG\times _G X$.  
There is an associated
spectral sequence with $E_2$-term $$E_2 ^{i, j} = H^i (G; H^j (X ) )$$
that converges to $H^{i+j} (X_{hG})$.  

Note that since $G$ acts freely 
on $X$ we have $X_{hG} \simeq X/G$. From this one obtains 
that the cohomology ring $H^* (X_{hG})$ is  finite-dimensional in each degree and vanishes above some degree. The first nonzero differential in the above spectral sequence takes the 
generators of $H^{2k-1} (X ) =\bF_p \oplus \bF_p$ to the cohomology classes $\mu_1, \mu_2 $ in $H^{2k} (G)$.
These classes are called the $k$--invariants of the $G$--space $X$. 

For any subgroup $H \leq G$, we  can restrict the above spectral sequence to the one for the
 action of $H$ on $X$. This follows from the fact that the Borel construction is natural. The   $k$--invariant of this restricted action will be $Res ^G _ H \mu_1$ and $Res ^G _H \mu_2$, where $$Res ^G _H : H^* (G ) \to H^* (H)$$ denotes the homomorphism induced by inclusion of $H$ into $G$.

Let $V$ denote the (unique) normal elementary abelian subgroup $\bZ/p\times \bZ/p$ in $G$. Let $\theta _1, \theta _2$ denote the $k$-invariants of the restricted $V$-action on $X$. Note that the classes $\theta _i$ are restrictions of cohomology classes $\mu_1, \mu_2$ in $H^{2k} (G )$.
By the Cartan-Eilenberg stable element theorem, the classes $\theta _i$ lie in the invariant subring $H^* (V ) ^{SL _2 (p)}$. This invariant ring is described in detail in \cite[Proposition 1.4.1, Claim 1.4.2]{Long}. If we write $$H^* (V)=\bF_p [x, y]\otimes \wedge (u, v),$$  then $H^{ev} (V) ^{SL_2 (p)}=\bF_p [x , y]^{SL_2 (p)} \otimes \wedge (vu)$   where $\bF_p [x, y] ^{SL_2(p)} =\bF_p [ \xi, \zeta] $ is a polynomial subalgebra generated by 
$$\xi= \sum _{i=0} ^p (x ^{p-i} y ^i )^{p-1}   \quad \text{and} \quad \zeta= x y^p-y x^p.$$

For $i=1,2$, let  $\theta_i=f_i(\xi, \zeta)+ uvg_i(\xi, \zeta)$ for some polynomials $f_i, g_i$.  Since $\theta _1$ and $\theta_2$ are integral classes, i.e., they are in the image of the map $H^* (V, \bZ) \to H^* (V, \bF_p)$ induced by mod-$p$ reduction, we have $g_i=0$ for $i=1,2$. This can be seen easily by applying the Bockstein operator  $\beta: H^* (V )\to H^{*+1} (V )$ to the classes $\theta_i$. Since $\beta(u)=x$ and  $\beta(v)=y$, we obtain  $$0=\beta (\theta _i)=\beta(uv)g_i=(xv - uy)g_i.$$ This gives $g_i=0$.  Hence the $k$-invariants $\theta _1, \theta _2$ lie in the polynomial subalgebra $\bF_p[\xi, \zeta]$.

Let us write $I$ for the ideal in $H^* (V)$ generated by $\theta_1$ and $\theta_2$. By a theorem of Carlsson \cite[Corollary 7]{Car82}, the cohomology ring $H^*(X_{hV})\cong H^*(X/V) $ is isomorphic to  $H^* (V)/I $ and the ideal $I$ is closed under Steenrod operations. For such an ideal we prove the following.

\Pro{\label{pro:generated} Let $M \subset H^{2k} (V )$ be a nonzero subgroup which lies inside the subalgebra $R[\xi, \zeta]$. Assume that the ideal $H^* (V ) \cdot M$ generated by $M$ is  stable under the Steenrod operations. Then $p+1$ divides $k$ and $M=\{ \lambda \zeta ^{k/(p+1)} \, |\,  \lambda  \in \bF_p \}$. 
}

\Proof{We will use an argument similar to the argument given by Oliver \cite{Oli79} for $A_4$-actions on $\sS^n \times \sS^n$.
We can write a nonzero element $m$ in $M$ as a sum $m=(f_0 (\xi)+ f_1 (\xi)\zeta+ \cdots + f_t (\xi) \zeta^t) \zeta ^s$ for some homogeneous polynomials $f_i (\xi)$, where $f_0$ and $f_t$ are nonzero (here $s$ is possibly zero). We claim that $f_i=0$ for all positive $i$ and $f_0$ is a scalar. By direct calculation, it is easy to see that $P^1 (\zeta)=0 $ and $P^1 (\xi)=\zeta ^{p-1}$. Note that the degree of $\xi$ is $2p(p-1)$, the degree of $\zeta$ is $2(p+1)$, and $P^1$ increases degree by $2(p-1)$. Since $M$ generates an ideal closed under Steenrod operations,  $P^1 (m)=\sum _i \alpha_i m_i$ for some $m_i \in M$ and $\alpha_i \in H^*(V)$. The elements in $M$ are invariant under $SL_2(p)$-action, so the coefficients $\alpha _i$ are also invariant. But there are no $2(p-1)$ dimensional classes in the invariant subring. So, we must have $P^1 (m)=0$.  

We claim that any homogeneous polynomial $m=(f_0 (\xi)+ f_1 (\xi)\zeta+ \cdots + f_t (\xi) \zeta^t) \zeta ^s$ satisfying $P^1(m)=0$ is a scalar multiple of a power of $\zeta$. We will prove this by downward induction on the degree of $m$. First note that if $s\neq 0$, then $m=m'\zeta $ for some $m'$, and $P^1 (m)=0$ gives $P^1(m')=0$. Since $m'$ is a polynomial with smaller degree, by induction $m'$ is a scalar multiple of a power of $\zeta$, hence $m$ also has the same property. So, we can assume $s=0$ in the above formula. 

By applying $P^1$ to $m$, we get 
$$\sum _{i=0} ^t  P^ 1( f_i ) \zeta ^{i}=0$$
that gives $P^1 (f_i )=0$ for all $i$.  If $f(\xi)=a \xi ^d\neq 0$ is an homogeneous polynomial in $\xi$ such that $P^1 (f )=0$, then we have $$0=P^1 (f)=aP^1 (\xi ^d )=ad \xi ^d \zeta ^{p-1}$$ that implies that $d=pd'$ for some $d'$. Hence $f(\xi)=[g(\xi)] ^p$ where $g(\xi)=a\xi^{d'}$. Writing $$m= \sum _i a_i \xi^{d_i p} \zeta ^{i}=a_0 \xi ^{d_0p} + a_1 \xi ^{d_1 p} \zeta + \cdots + a_t  \xi ^{d_t p} \zeta ^t, $$ we see that for such a polynomial to be a homogeneous polynomial, we must have $a_i=0$ for all $i$ which are not multiples of $p$. This is because the degree of $\xi ^{d_i p} \zeta ^{i}$ is $2i$ mod $p$. Hence we can write $m$ in the form 
$$m=\Bigl ( \sum _i a_i \xi ^{d_i} \zeta ^i \Bigr ) ^p $$ by reindexing the sum. This means that to each $m$ of the above form, we can associate the element $m'= \sum _i a_i \xi ^{d_i} \zeta ^i $ which  has strictly smaller degree. The $M'$ generated by these $m'$ elements satisfies the conditions of the proposition because being closed under Steenrod operations is a property of the radical of the ideal \cite[Corollary 8]{Car82}, so taking the $p$-th root does not effect this property. Therefore by induction $M'$ is generated by a power of $\zeta$. Hence $M$ is also generated by a power of $\zeta$.
 }

We are now ready to complete the proof of Theorem \ref{thm:noprodaction}.

\Proof{[Proof of Theorem \ref{thm:noprodaction}] Let $I=(\theta _1, \theta _2)$ denote the ideal generated by the k-invariants  of the $V$-action on $X \simeq \sS^n \times \sS^n$.  By \cite[Corollary 7]{Car82}, there is an isomorphism $$H^* (X/V ) \cong H^* (V) /I$$ hence $H^* (V) /I$ (ungraded) is a finite-dimensional vector space.    Moreover the ideal $I$ is closed under  Steenrod operations and it is of the form $M\cdot H^* (V)$ where $M\subset \bF_p [x,y]^{SL_2 (p)}$. Hence by  Proposition \ref{pro:generated} the ideal $I$ is generated by $\zeta ^{k/(p+1)}$. This gives a contradiction because the fact that $H^* (V)/I$ is finite-dimensional implies that $I$ cannot be generated by one element by standard results in commutative algebra (see \cite[Proposition 3]{Car82}).} 
 
\rm{
\section{Appendix}\label{appendix}  
By a space we mean either a topological space or a simplicial set. The relation between the two is given by the singular simplicial set functor and the geometric realization functor. The category of simplicial sets is a model category with Quillen model structure with the usual weak equivalences and Kan fibrations. The geometric realization functor carries a Kan fibration to a Serre fibration.

\subsection{Mapping spaces}\label{mapping_space}
Let $X$ and $Y$ be simplicial sets. 
The mapping space $\Map(X,Y)$ is the simplicial set whose set of $n$-simplices is given by $\catS(\Delta[n]\times X,Y)$, and the simplicial structure is induced by the ordinal maps $\Delta[n]\rightarrow \Delta[m]$. $\Map_B(X,Y)_f$ will denote the connected component of a map $f:X\rightarrow Y$, in other words the space of maps which are homotopic to $f$. Let $\catS_{/B}$ denote the over category whose objects are maps $X\rightarrow B$ and whose morphisms are commutative triangles over $B$. Let  $\Map_{B}(X,Y)$  denote the mapping space for the over category which is defined to be the simplicial set with $n$-simplices given by the set  $\catS_{/B}(\Delta[n]\times X,Y)$ with the simplicial structure defined similarly.  

Let $\xi:E\rightarrow B$ be a Kan fibration of simplicial sets. Then $\xi$ induces a fibration of mapping spaces
$$
\xi_*:\Map(B,E)\rightarrow \Map(B,B)
$$
whose fiber  over the identity map $\idy:B\rightarrow B$ is  the mapping space $\Map_B(B,E)$. We will also denote the fiber by $\Sec(\xi)$ and call it the space of sections of $\xi$. 
There is a pull-back diagram
$$
\begin{tikzcd}
\hSec(\xi) \arrow{r} \arrow{d} & \Map(B,E) \arrow{d}{\xi_*} \\
\Map(B,B)_\idy \arrow{r} & \Map(B,B)
\end{tikzcd}
$$
where $\hSec(\xi)$ is the space of maps $B\rightarrow E$ such that the triangle
$$
\begin{tikzcd}
B \arrow{r} \arrow[d,equal] & E \arrow{dl}{\xi} \\
B 
\end{tikzcd}
$$
commutes up to homotopy.

In Lannes' theory we will consider fibrations over the classifying space  of an elementary abelian $p$-group $V$.
The classifying space $BV$ is a simplicial group with product $BV\times BV\rightarrow BV$ induced by the product on $V$.   The mapping space $\Map(BV,BV)$ is isomorphic to $\Hom(V,V)\times BV$ as a simplicial set \cite[Proposition 25.2]{May67}, and the adjoint $BV\rightarrow \Map(BV,BV)$ of the product  map  identifies $BV$ with the identity component of the mapping space.  For any  $X$ the mapping space $\Map(BV,X)$ has an induced action of $BV$. Also the simplicial monoid $\Map(BV,BV)$ acts by pre-composition on the mapping space.  The actions are equivariant with respect to the isomorphism $BV\rightarrow \Map(BV,BV)_\idy$ of simplicial abelian groups \cite[Proposition 25.3]{May67}. Given a fibration $\xi:E\rightarrow BV$ in the diagram of fibrations
$$
\begin{tikzcd}
BV\times \Sec(\xi) \arrow{r} \arrow{d} & \hSec(\xi) \arrow{d} \\
BV \arrow{r} & \Map(BV,BV)_\idy
\end{tikzcd}
$$
the horizontal arrows are isomorphisms.

\subsection{Homotopy fixed points}\label{homotopy_fix} Let $G$ be a discrete group   and $X$ be a simplicial set with $G$-action, also called a $G$-space. The homotopy orbit space $X_{hG}$ is the quotient $(EG\times X)/G$ under the diagonal action. The homotopy fixed point space $X^{hG}$ is the simplicial subset  $\Map(EG,X)^G$  of $G$-equivariant simplicial set maps in $\Map(EG,X)$. Let $f:X\rightarrow Y$ be a map of $G$-spaces which is also a weak equivalence. Then the induced maps $X_{hG}\rightarrow Y_{hG}$ and $X^{hG}\rightarrow Y^{hG}$ are weak equivalences. 

 Next we describe a transitivity property of homotopy fixed points \cite[Lemma 10.5]{DW94}. Let $H$ be a normal subgroup of $G$. There is a natural action of $G/H$ on the mapping space $\Map(EG,X)^H$. 
 Then there is a weak equivalence
$$
X^{hG}\simeq (\Map(EG,X)^H)^{hG/H}  
$$
where $\Map(EG,X)^H\simeq X^{hH}$.

Let $\xi:E\rightarrow BG$ be a fibration with fiber $F$. Consider the pull-back diagram
$$
\begin{tikzcd} 
X_\xi\arrow{d}\arrow{r}&E\arrow{d}{\xi}\\
EG\arrow{r}&BG
\end{tikzcd}
$$
along the universal principal $G$-fibration.    Since $X_\xi$ is a free $G$-space, the natural map $h:(X_\xi) _{hG}\rightarrow (X_\xi) /G=E$ is a weak equivalence, in fact a trivial fibration.  There is a map of fibrations
$$
\begin{tikzcd}
X_\xi \arrow{r}\arrow{d}&F\arrow{d}\\
(X_\xi)_{hG}\arrow{d}{\tilde \xi}\arrow{r}{h}&E\arrow{d}{\xi}\\
BG\arrow[r,equal]&BG
\end{tikzcd}
$$
where $\tilde \xi$ is the composition $\xi \circ h$. We will usually switch from an arbitrary fibration to  the natural projection $(X_\xi) _{hG}\rightarrow BG$. As a consequence we have the following identifications
$$
(X_\xi) ^{hG}=\Map(EG, X_\xi )^G= \Sec(\tilde \xi)\simeq\Sec(\xi) .
$$
The equivalence $\Sec(\tilde \xi)\simeq\Sec(\xi)$ is a consequence of the fact that the map
$$
\Map(BG,Y) \rightarrow \Map(BG,Y')
$$
induced by a trivial fibration $Y\rightarrow Y'$ is also a trivial fibration \cite{GJ99}. Note that both spaces are fibrations over $\Map(BG,BG)$. Pulling back along the subspace $\Map(BG,BG)_\idy\rightarrow \Map(BG,BG)$ induces the required weak equivalence between the spaces of sections. 

\subsection{$hG$--equivalence}\label{hG}
Let $\xi_i:E_i \rightarrow BG$ for $i=1,2$ be two fibrations. A map $E_1\rightarrow E_2 $ of fibrations over $BG$ is called a fiber homotopy equivalence if there is a homotopy inverse over $BG$. Let $X_{\xi_i}$ denote the associated $G$--spaces. Then $f$ induces a map $X_{\xi_1}\rightarrow X_{\xi_2}$ which is a $G$--homotopy equivalence. Conversely one can start with two $G$--spaces $X_1$ and $X_2$ and compare the fibrations associated to the Borel constructions. In this case a weaker notion of equivalence is enough. We say $X_1$ and $X_2$ are $hG$--equivalent if there is a zig-zag of $G$--maps which are also weak equivalences. Given $hG$--equivalent $G$--spaces $X_1$ and $X_2$ the products $EG\times X_1$ and $EG\times X_2$ are $G$--equivalent. Therefore $(X_1)_{hG}$ is fiber homotopy equivalent to $(X_2)_{hG}$. This implies that there is a one-to-one correspondence between fiber homotopy classes of fibrations over $BG$ and $hG$--equivalence classes of $G$--spaces.

\subsection{Completion at a prime} \label{pre_completion}
Let $X^\wedge _p$ denote the Bousfield-Kan completion of $X$ at a prime $p$ as defined in  \cite{BK72}. 
It comes with a natural map $X\rightarrow X^\wedge_p$. A space is called $p$--complete if this map is a weak equivalence. For example, the classifying space $BP$ of a $p$-group is $p$--complete.
A map $f:X\rightarrow Y$ induces an isomorphism $\tilde{H}_*(f,\FF_p)$ if and only if its $p$--completion $f^\wedge_p:X^\wedge_p \rightarrow Y^\wedge_p$  is a weak equivalence. 
Note that any weak equivalence between $p-$completed spaces is a homotopy equivalence since the $p-$completion of a space is a fibrant simplicial set, i.e. a Kan complex.  
Let $\xi:E\rightarrow B$ be a fibration with fiber $F$. 
The $p$--completion $\xi^\wedge_p:E^\wedge_p \rightarrow B^\wedge_p$ is still a fibration.   The fiber lemma \cite[Chapter II 5.1]{BK72} implies that if $\pi_1 B$ is a $p$--group and $F$ is connected then the fiber of $\xi^\wedge_p$ is the $p$--completion of $F$.
There is also a relative version of the completion construction which applies to a fibration $\xi:E\rightarrow B$,  called the fiberwise completion at a prime $p$. We will denote the fiberwise $p$-completion  by $\xi^\wedge_{p/B}:E^\wedge_{p/B}\rightarrow B$. This is  a fibration whose fiber is given by $F^\wedge_p$. 
If $\xi$ is a fibration over $BP$ then up to homotopy $\xi^\wedge_p$ can be identified with $\xi^\wedge_{p/B}$. More explicitly, there is a map of fibrations
$$
\begin{tikzcd}
E^\wedge_{p/BP} \arrow{r} \arrow{d} & E^\wedge_{p} \arrow{d}\\
BP \arrow{r} & BP^\wedge_p
\end{tikzcd}
$$
where the horizontal maps are weak equivalences.

\begin{pro}\label{comp_cover} 
Let $\xi:E\rightarrow BP$ be a   fibration  with connected fiber  and    let $X=(X_\xi) _p ^\wedge$ denote the $p$-completion of the $P$-space associated to $\xi$. Then there is a fiber homotopy equivalence
$$
\begin{tikzcd}
X_{hP} \arrow{r}{\sim} \arrow{d} & E^\wedge_p \arrow{d} \\
BP \arrow[r,equal] & BP
\end{tikzcd}
$$ 
\end{pro}
\begin{proof} This result is proved for an elementary abelian $p$--group in \cite[Proposition 4.3.1]{Lan75}. For a general $p$--group we proceed as follows. 
Let $X_\xi$ denote the pull-back of $\xi$ along $EP\rightarrow BP$. It fits into a fibration sequence $F\rightarrow X_\xi \rightarrow EP$. After completion the sequence $F^\wedge_p\rightarrow (X_\xi)^\wedge_p \rightarrow EP^\wedge_p$ is a fiber sequence since $EP^\wedge_p$ is still contractible. Then  the diagram
$$
\begin{tikzcd}
(X_\xi)^\wedge_p  \arrow{r}\arrow{d} & EP^\wedge_p \arrow{d} \\
 E^\wedge_p \arrow[r] & BP^\wedge_p
\end{tikzcd}
$$
is a homotopy pull-back diagram since $E^\wedge_p\rightarrow BP^\wedge_p$ is a fibration with fiber $F^\wedge_p$ by the fiber lemma. This implies that $X=(X_\xi)^\wedge_p$ is a $P$--covering of $E^\wedge_p$.
Thus $X$ is a free $P$--space and the desired map is given by the projection $X_{hP}\rightarrow X/P$.
\end{proof}

\subsection{Fiber joins} \label{pre_fibjoin}
Next we will discuss the fiber join construction. First we look at the behaviour with respect to Borel construction and $p$--completion, and then we study the Euler class of fiber joins. 
The fiber join of two fibrations $\xi_1:E_1\rightarrow B$ and $\xi_2:E_2\rightarrow B$ is defined to be the homotopy push-out of
\begin{equation}\label{def_fibjoin}
\begin{tikzcd}
E_1\times_B E_2 \arrow{r} \arrow{d} & E_1 \arrow{d} \\
E_2 \arrow{r} & E_1\join_BE_2
\end{tikzcd}
\end{equation}
where $E_1\times_B E_2$ is the pull-back of the maps $\xi_1$ and $\xi_2$ over $B$. When $B$ is a point this definition specializes to the join construction, and we simply write $E_1\ast E_2$.

\begin{pro}\label{Borel_join}
Let $X$ and $Y$ be $G$--spaces.
Then there is a fiber homotopy equivalence
$$
\begin{tikzcd}
X_{hG} \ast_{BG} Y_{hG} \arrow{r}{\sim} \arrow{d} & (X \ast Y)_{hG} \arrow{d} \\
BG \arrow[r,equal] & BG\ .
\end{tikzcd}
$$
\end{pro}
\Proof{
The map is induced by the Borel construction of the natural maps
$$
X \rightarrow X\ast Y \leftarrow Y
$$
and restricts to the identity map between the fibers.
}

Given two spaces $X,Y$ the join $X^\wedge_p \ast Y^\wedge_p$ is in general not $p$--complete   \cite[pg. 107]{Sul70}. 
For example, for spheres the $p$-completion of $(\sS^n)^\wedge_p \ast (\sS^m)^\wedge_p$ is given by $(\sS^n\ast \sS^m)^\wedge_p$. We use the following notation
$$
X[m]= (\underbrace{X\ast\cdots \ast X}_m)^\wedge_p.
$$
A fiberwise  version of this definition is as follows: Given a fibration $\xi:E\rightarrow B$ we define 
$$
E_{/B}[m]=(\underbrace{E \ast_B\cdots \ast_B E}_m)^\wedge_{p/B}
$$ 
and denote the associated fibration by $\xi[m]:E_{/B}[m]\rightarrow B$. 

\Cor{\label{not_equiv} 
Let $P$ be a $p$--group. If $\xi$ is the fibration $E\rightarrow BP$ associated to the Borel construction of a $P$--space $X$, then $\xi[m]$ is fiber homotopy equivalent to $(X[m])_{hP}$.
} 
\Proof{
Using Proposition \ref{Borel_join} and induction we see that there is a fiber homotopy equivalence
$$
\ast^n_{BP} E \rightarrow (\ast^n X)_{hP}
$$
over $BP$. Let $Y$ denote the pull-pack of the fiber join $\ast^n_{BP} E$ along $EP\rightarrow BP$. Then $Y$ is $G$--equivalent to $\ast^n X$. Applying Proposition \ref{comp_cover} to $\ast^n X$ we obtain a fiber homotopy equivalence
$$
((\ast^n X)^\wedge_p)_{hP} \rightarrow (\ast^n_{BP} E)^\wedge_p.
$$
} 

\Rem{\rm{\label{not_pro} It is useful to observe that the natural map $ X[mn]\rightarrow (X[m])[n]$ 
is a homotopy equivalence. Similarly $\xi[mn]$ is fiber homotopy equivalent to $(\xi[m])[n]$.
}}

\subsection{Euler class}  
We will study the Euler class of fiber joins of mod-$p$ spherical fibrations. Note that the fiber join construction does not result in a mod-$p$ spherical fibration until we fiberwise complete it at $p$. But the resulting fibration has a fiber whose mod-$p$ cohomology is the mod-$p$ cohomology of a sphere. 
Therefore for our purposes we consider a larger class of fibrations. Let   $\xi:E\rightarrow B$ be a fibration where the fiber $F$ satisfies $HF\cong H(\sS^d)$ for some $d$.
Thom space $\Th(\xi)$ of $\xi$ is defined to be the cofiber of the map $\xi$.  
 Consider  the diagram of cofibrations
$$
\begin{tikzcd}
F \arrow{r} \arrow{d} & E \arrow{d}{\xi}\\ 
\ast \arrow{r} \arrow{d} & B \arrow{d} \\
\Sigma F\arrow{r}{\mu} & \Th(\xi) \ .
\end{tikzcd}
$$
The reduced cohomology  ring $\tilde H(\Th(\xi))$ is a free $HB$--module generated on a generator $t(\xi)$ of degree $d+1$ called the Thom class of $\xi$. We can take $t(\xi)$ to be the dual of the image of a chosen generator under the map $H_{d+1}(\Sigma F)\rightarrow H_{d+1}(\Th(\xi))$ induced by $\mu$.
 The image of the Thom class under the natural map $H^{d+1}(\Th(\xi))\rightarrow H^{d+1}(B)$ is called the Euler class $e(\xi)$. Alternatively $e(\xi)$ is the image of the transgression of the generator of $H^d(F)$ in the Serre spectral sequence of $\xi$.  

We will need an  alternative description of fiber join construction. Given fibrations $\xi_i:E_i\rightarrow B$ with fiber $F_i$ where $i=1,2$ let us define a quotient space
$$E_1\hat\join E_2=E_1\times E_2\times [0,1]/\sim$$ 
by the relations $(e_1,e_2,0)\sim (e_1,e_2',0)$ if $\xi_2(e_2)=\xi_2(e'_2)$ and  $(e_1,e_2,1)\sim (e_1',e_2,1)$ if $\xi_1(e_1)=\xi_1(e'_1)$. Projecting onto each factor induces   a fibration $\xi_1 \hat \join \xi_2: E_1\hat\join E_2 \rightarrow B\times B$ with fiber  $F_1\join F_2$. The pull-back along the diagonal inclusion $\Delta:B\rightarrow B\times B$ is exactly the fiber join construction
$$
\begin{tikzcd}
F_1\join F_2 \arrow[r,equal] \arrow{d} & F_1\join F_2\arrow{d}\\
 E_1\join_BE_2 \arrow{r}\arrow{d}& E_1\hat\join E_2 \arrow{d}{\xi_1 \hat \join \xi_2} \\
B\arrow{r}{\Delta} &B\times B
\end{tikzcd}
$$
In more details, the diagram \ref{def_fibjoin} maps to the pull-back and  induces a weak equivalence between the spaces given in the two definitions.

\begin{pro}\label{fiberjoin} Let $\xi_1$ and $\xi_2$ be mod-$p$ spherical fibrations over $B$.
In the commutative diagram
$$
\begin{tikzcd}
\Sigma(F_1\ast F_2) \arrow{r}{\sim} \arrow{d} & \Sigma F_1\sm \Sigma F_2 \arrow{d}{\Sigma\mu_1\wedge\Sigma\mu_2} \\
\Th(\xi_1 \hat \join \xi_2) \arrow{r}{\sim} & \Th(\xi_1)\sm\Th(\xi_2)
\end{tikzcd}
$$
 the horizontal maps are weak equivalences.
\end{pro}
\Proof{ Let $D(\xi)$ denote the  fiber join of the identity map $B\rightarrow B$ (regarded as a fibration) with a mod-$p$ spherical fibration $\xi:E\rightarrow B$. The Thom space $\Th(\xi)$ can be described as the quotient $D(\xi)/B$.
We have the following identifications
$$
\Th(\xi_1)\sm\Th(\xi_2)=\frac{\Th(\xi_1) \times \Th(\xi_2)}{\Th(\xi_1)\vee\Th(\xi_2)} = \frac{D(\xi_1)\times D(\xi_2)}{(E\times D(\xi_2))\cap (D(\xi_1)\times E )} \simeq \frac{D(\xi_1\hat\join\xi_2)}{\xi_1\hat\join\xi_2} =T(\xi\hat\join\xi)
$$
which is compatible with the equivalence $\Sigma F_1\sm \Sigma F_2 \simeq \Sigma(F_1\ast F_2)$.
}

An almost immediate consequence of Proposition \ref{fiberjoin} is that the Euler class of the fiber join of two mod-$p$ spherical fibrations is the cup product of Euler classes of the individual fibrations.  

\Cor{\label{prod_Euler} If $\xi_1$ and $\xi_2$ are mod-$p$ spherical fibrations, then
$$
e(\xi_1\join_B \xi_2)=e(\xi_1)e(\xi_2).
$$
}
\begin{proof}  Proposition \ref{fiberjoin} implies that the Thom class of $\xi_1 \hat \join \xi_2$ is  the cross product of the Thom classes of $\xi_1$ and $\xi_2$. Looking at the corresponding diagram of cohomology groups associated to the diagram   
$$
\begin{tikzcd}
B\arrow{r}{\Delta} \arrow{d} & B\times B \arrow{d}\\
\Th(\xi_1\join_B\xi_2) \arrow{r} & \Th(\xi_1)\sm\Th(\xi_2)\ .
\end{tikzcd}
$$
We see that the Euler class of $\xi_1\join_B \xi_2$ is the cup product of the Euler classes of $\xi_1$ and $\xi_2$. 
\end{proof}
 }

\bibliographystyle{plain}

\end{document}